\documentclass[12pt]{amsart}  

\usepackage{custom_preamble} 

\begin{document}

\title{On the Bogolyubov-Ruzsa lemma}

\author{\tsname}
\address{\tsaddress}
\email{\tsemail}

\begin{abstract}
Our main result is that if $A$ is a finite subset of an Abelian group with $|A+A| \leq K|A|$, then $2A-2A$ contains an $O(\log^{O(1)}2K)$-dimensional coset progression $M$ of size at least $\exp(-O(\log^{O(1)}2K))|A|$.
\end{abstract}

\maketitle

\section{Introduction}

In the recent paper \cite{crosis::}, Croot and Sisask introduced a fundamental new method to additive combinatorics and, although they have already given a number of applications, our present purpose is to give another.  Specifically, we shall prove the following.
\begin{theorem}[Bogolyubov-Ruzsa lemma for Abelian groups]\label{thm.main}
Suppose that $G$ is an (discrete) Abelian group and $A,S \subset G$ are finite non-empty sets such that $|A+S| \leq K\min\{|A|,|S|\}$.  Then $(A-A)+(S-S)$ contains a proper symmetric $d(K)$-dimensional coset progression $M$ of size $\exp(-h(K))|A+S|$.  Moreover, we may take $d(K)=O(\log^62K)$ and $h(K)=O(\log^62K\log2\log 2K)$.
\end{theorem}
We should take a moment to justify the name, which is slightly non-standard.  Bogolyubov's lemma (the idea for which originates in \cite{bog::}) is usually stated for sets of large density in the ambient group, rather than small doubling, and asserts that the four-fold sumset of a thick set contains a large Bohr set.

Ruzsa, on his way to proving Fre{\u\i}man's theorem in \cite{ruz::9}, showed that a set with small doubling could be sensibly embedded into a group where it is thick.  He then applied Bogolyubov's lemma and proceeded to show that a Bohr set contains a large generalised arithmetic progression which could then be pulled back.  In doing all this he implicitly proved the first version of Theorem \ref{thm.main} in $\Z$ -- although, with different bounds -- and this motivates the name.

This result has many variants (although the form given above seems to be a fairly useful one) and in light of this the history is not completely transparent.  Certainly most proofs of Fre{\u\i}man's theorem broadly following the model of \cite{ruz::9} will implicitly prove a result of this shape.  With this in mind the extension from $\Z$ to arbitrary Abelian groups is due to Green and Ruzsa \cite{greruz::0}, and the first good bounds to Schoen \cite{sch::1} for certain classes of groups.

There are many applications of results of this type, particularly since their popularisation by Gowers \cite{gow::4}, and we shall deal with a number of these in \S\ref{sec.appln} at the end of the paper.  To help explain the main ideas we include a discursive sketch of the paper after the next section, which simply sets some notation.

\section{Notation}

The main tool used in the paper is Fourier analysis on groups for which the classic reference is Rudin \cite{rud::1}.  We deal almost exclusively with finite groups in the paper, but to be complete we shall need slightly more generality.

Suppose that $G$ is a locally compact topological group.  We write $C(G)$ for the space of continuous complex-valued functions on $G$.  More generally if $R \subset \C$ we write $C(G,R)$ for the continuous $R$-valued functions on $G$.

The group structure on $G$ induces an action of $G$ on $C(G)$ called translation.  In particular if $x \in G$ and $f \in C(G)$ then we write
\begin{equation}\label{eqn.rhodef}
\rho_x(f)(y):=f(yx) \textrm{ for all } y \in G.
\end{equation}
We also write $M(G)$ for the space of regular Borel measures on $G$ and can extend $\rho$ to these in the natural way: for $x \in G$ and $\mu \in M(G)$, $\rho_x(\mu)$ is the measure induced by
\begin{equation*}
C(G) \rightarrow C(G); f \mapsto \int{f(x)d\mu(yx)}.
\end{equation*}
The group structure on $G$ is reflected in $M(G)$ in a fairly natural way and we define the convolution of two measures $\mu,\nu \in M(G)$ to be the measure $\mu \ast \nu$ induced by
\begin{equation*}
C(G) \rightarrow C(G); f \mapsto \int{f(xy)d\mu(x)d\nu(y)}.
\end{equation*}
There is a family of privileged measures on $G$ called Haar measures.  These are the translation invariant measures on $G$: $\mu\in M(G)$ is a Haar measure on $G$ if $\rho_x(\mu)=\mu$ for all $x \in G$.

Given a Haar measure $\mu$ on $G$ we can extend $\rho$ in the obvious way from (\ref{eqn.rhodef}) to define the right regular representation $\rho:G \rightarrow \Aut(L^2(\mu))$.  More than this we can define the convolution of two functions $f,g \in L^1(\mu)$ by
\begin{equation*}
f \ast g(x):= \int{f(y)g(y^{-1}x)d\mu(y)} \textrm{ for all }x \in G.
\end{equation*}

There are two particularly useful instances of Haar measure depending on the topology on $G$: if $G$ is compact we write $\mu_G$ for the Haar probability measure on $G$, while if $G$ is discrete we write $\delta_G$ for the Haar counting measure on $G$ and assigns mass $1$ to each element of $G$.

Of course, if $G$ is finite it is both discrete and compact so one has both probability measure and counting measure to choose from.  The measures are multiples of each other as $\mu_G$ is just the measure assigning mass $|G|^{-1}$ to each element of $G$.  More generally given a finite set $X$ we write $\mu_X$ for the measure assigning mass $|X|^{-1}$ to each $x \in X$.

When it is relevant we shall indicate whether we are taking a finite group $G$ to be compact or discrete by declaring the group either compact, so that $\mu_G$ is to be used, or discrete so that $\delta_G$ is to be used.  The reader should be aware that this has the effect of changing the normalisations in convolutions.

The above all works for general finite groups $G$, but when $G$ is also Abelian convolution operators can be written in a particularly simple form with respect to the Fourier basis which we now recall.

We write $\wh{G}$ for the dual group, that is the finite Abelian group of homomorphisms $\gamma:G \rightarrow S^1$, where $S^1:=\{z \in \C:|z|=1\}$.  Given $\mu \in M(G)$ we define $\wh{\mu} \in \ell^\infty(\wh{G})$ by
\begin{equation*}
\wh{\mu}(\gamma):=\int{\overline{\gamma}d\mu} \textrm{ for all } \gamma \in \wh{G},
\end{equation*}
and extend this to $f \in L^1(\mu_G)$ by $\wh{f}:=\wh{fd\mu_G}$. It is easy to check that $\wh{\mu \ast \nu} = \wh{\mu}\cdot \wh{\nu}$ for all $\mu,\nu \in M(G)$ and $\wh{f \ast g} = \wh{f} \cdot \wh{g}$ for all $f,g \in L^1(\mu_G)$.

\section{A sketch of the argument}

Assuming the hypotheses of Theorem \ref{thm.main} our objective will be to show that there is a large, low-dimensional coset progression $M$ correlated with $A+S$, meaning such that
\begin{equation*}
\|1_{A+S}\ast \mu_M\|_{\ell^\infty(G)} > 1-o(1).
\end{equation*}
This is essentially the statement of Theorem \ref{thm.core} later, and Theorem \ref{thm.main} can be derived from it by a simple pigeonholing argument.

\subsection*{A simplified argument: the case of good modelling}  We shall assume that we have good modelling in the sense of \cite{greruz::0}, meaning that we shall assume that the sets $A$ and $S$ have density $K^{-O(1)}$ in the ambient group.  This can actually be arranged in the two cases of greatest interest: $\F_2^n$ and $\Z$ and facilitates considerable simplifications.

A very useful observation in L{\'o}pez and Ross \cite{lopros::} is that because the support of $\mu_A \ast \mu_S$ is contained in $A+S$ we have the identity
\begin{equation*}
\langle 1_{A+S}\ast \mu_{-S}, \mu_A \rangle = 1.
\end{equation*}
Now, suppose we had a coset progression $M$ over which $1_{A+S} \ast \mu_{-S}$ was in some sense invariant, meaning
\begin{equation}\label{eqn.uq}
 \|1_{A+S} \ast \mu_{-S} \ast \mu_M - 1_{A+S} \ast \mu_{-S}\|_{\ell^p(G)} \leq \epsilon\|1_{A+S}\|_{\ell^p(G)}.
\end{equation}
Then H{\"o}lder's inequality and the L{\'o}pez-Ross identity tell us that
\begin{equation*}
|\langle 1_{A+S} \ast \mu_{-S} \ast \mu_M,\mu_A \rangle - 1| \leq \epsilon  \|1_{A+S}\|_{\ell^p(G)}\|\mu_A\|_{\ell^{p/(p-1)}(G)} \leq \epsilon K^{1/p},
\end{equation*}
and it follows by averaging that $A+S$ is correlated with $M$ provided that $\epsilon \sim K^{-1/p}$.

The traditional Fourier analytic approach to finding an $M$ such that (\ref{eqn.uq}) holds is not particularly efficient, but recently Croot and Sisask showed that there is, at least, a set $Z$ such that we have (\ref{eqn.uq}) with $Z$ in place of $M$ and
\begin{equation*}
\mu_G(Z) \geq \exp(-O(\epsilon^{-2}p\log K))\mu_G(A).
\end{equation*}
Moreover, they noted by the triangle inequality that one can endow $Z$ with the structure of a $k$-fold sumset, so that we have (\ref{eqn.uq}) with $kX$ in place of $M$ and
\begin{equation}\label{eqn.bq}
\mu_G(X) \geq \exp(-O(k^2\epsilon^{-2}p\log K))\mu_G(A) = \exp(-O(k^2\log^2K))\mu_G(A),
\end{equation}
where the third term is by optimising the choice of $p\sim \log K$ given that $\epsilon \sim K^{-1/p}$.  

What we actually end up with after all this is a set $X$ with density as described in (\ref{eqn.bq}) such that
\begin{equation}\label{eqn.ere}
\langle 1_{A+S} \ast \mu_{-S} \ast \mu_X^{(k)},\mu_A \rangle >1-o(1).
\end{equation}
Now, by the usual sorts of applications of Plancherel's theorem and Cauchy-Schwarz we find that most of the Fourier mass of the inner product is concentrated on those characters in $\Spec_{1/2}(1_X)$ provided $2^k \sim K$, and so we choose $k \sim \log K$. 

With most of the Fourier mass supported on $\Spec_{1/2}(1_X)$, it follows that the integrand in (\ref{eqn.ere}) correlates with any set which approximately annihilates $\Spec_{1/2}(1_X)$.  It remains to show that the approximate annihilator of $\Spec_{1/2}(1_X)$ -- that is the Bohr set $B$ with $\Spec_{1/2}(1_X)$ as its frequency set -- contains a large coset progression.

We can now apply Chang's theorem to get that $B$ is low dimensional and then the usual geometry of numbers argument tells us that this Bohr set contains a large coset progression, and the result is proved.

\subsection*{Extending the argument: the case of bad modelling} We now drop the assumption of good modelling, and the argument proceeds in essentially the same way up until the application of Chang's theorem above.

In this case Chang's theorem does not provide good bounds.  Instead what we do is note that the set $X$ satisfies a relative polynomial growth condition
\begin{equation*}
|nX| \leq n^{O(\log^4K)}|X| \textrm{ for all } n \geq 1.
\end{equation*}
This lets us produce a Bohr set containing $X$ which behaves enough like a group for a relative version of Chang's theorem to hold, whilst at the same time $X$ is much denser in the Bohr set than it would be in the modelling group.

Since we are not using modelling what we have just done does not actually give us a Bohr set of low dimension, but rather a Bohr set of size comparable to $X$ which has a lower order of polynomial growth on a certain range.  It turns out that the usual argument that shows a low dimensional Bohr set contains a large coset progression can be adapted relatively easily to this more general setting and this gives us our final ingredient.

These arguments are spread over the paper as follows.  The simplified argument up to (\ref{eqn.ere}) is essentially contained in \S\ref{sec.bst}.  Then, in \S\ref{sec.bohrprop}, we record the basic properties of Bohr sets we need before \S\ref{sec.relchang}, which has the relative version of Chang's theorem, and \S\ref{sec.bohrcont}, which puts the material together to take a set satisfying a relative polynomial growth condition and produce a large Bohr superset.

After the material on Bohr sets we have \S\ref{sec.cov} which records some standard covering lemmas and then \S\ref{sec.lcp} where we show how to find a large coset progression in a Bohr set with relative polynomial growth.  Finally the argument is all put together in \S\ref{sec.pf}.

\section{Fre{\u\i}man-type theorems in arbitrary groups}\label{sec.bst}

In this section we are interested in Fre{\u\i}man-type theorems in arbitrary, possibly non-Abelian, groups.  There has been considerable work towards such results, although often with restrictions on the type of non-Abelian groups considered, or rather weak bounds.  We direct the reader to \cite{gre::05} for a survey, but our interest is narrower, lying with a crucial result of Tao \cite[Proposition C.3]{tao::9} which inspires the following.
\begin{proposition}\label{prop.bst}  Suppose that $G$ is a (discrete) group, $A,S \subset G$ are finite non-empty sets such that $|AS| \leq K\min\{|A|,|S|\}$, and $k \in \N$ is a parameter.  Then $A^{-1}ASS^{-1}$ contains $X^k$ where $X$ is a symmetric neighbourhood of the identity with size $\delta(k,K)|AS|$.  Moreover, we may take $\delta(k,K) \geq \exp(-O(k^2\log^22K))$.
\end{proposition}
Note that this result is a very weak version of Theorem \ref{thm.main} but for any group, not just Abelian groups, and despite its weaknesses, its generality makes it useful in some situations.

Croot and Sisask essentially prove the above result in \cite[Theorem 1.6]{crosis::} with weaker $K$-dependence in the bound, by using the $p=2$ version of their Lemma \ref{lem.cs} below.  It turns out that we shall be able to show the above bound by coupling the large $p$ case of their result with the L{\'o}pez-Ross identity.

The key proposition of this section then, is the following.
\begin{proposition}\label{prop.key}
Suppose that $G$ is (discrete) a group, $A,S,T \subset G$ are finite non-empty sets such that $|AS| \leq K|A|$ and $|TS| \leq L|S|$, and $k \in \N$ and $\epsilon \in (0,1]$ are a pair of parameters.  Then there is a symmetric neighbourhood of the identity $X \subset G$ with
\begin{equation*}
|X| \geq \exp(-O(\epsilon^{-2}k^2\log 2K \log 2L))|T|
\end{equation*}
such that
\begin{equation*}
|\mu_{A^{-1}} \ast 1_{AS}\ast \mu_{S^{-1}}(x) - 1| \leq \epsilon \textrm{ for all }x \in X^k.
\end{equation*}
\end{proposition}
The main ingredient in the proof of this is the following result which is essentially \cite[Proposition 3.3]{crosis::}.  To prove it Croot and Sisask introduced the idea of sampling from physical space rather than Fourier space -- sampling in Fourier space can be seen as the main idea in Chang's theorem.  Not only does this work in settings where the Fourier transform is less well behaved, but it also runs much more efficiently, which leads to the superior bounds.

We include the proof since it is the pivotal ingredient of this paper, and we frame it in such a way as to emphasise the parallels with Chang's theorem. 
\begin{lemma}[Croot-Sisask]\label{lem.cs}  Suppose that $G$ is a (discrete) group, $f \in \ell^p(G)$ for $p \geq 2$ and $S,T \subset G$ are non-empty with $|ST| \leq K|S|$.  Then there is a $t \in T$ and a set $X \subset Tt^{-1}$ with $|X| \geq (2K)^{-O(\epsilon^{-2}p)}|T|$ such that
\begin{equation*}
\|\rho_x(f \ast \mu_S) - f \ast \mu_S\|_{\ell^p(G)} \leq \epsilon \|f\|_{\ell^p(G)} \textrm{ for all } x \in X.
\end{equation*}
\end{lemma}
\begin{proof}
Let $z_1,\dots,z_k$ be independent uniformly distributed $S$-valued random variables, and for each $y \in G$ define $Z_i(y):=\rho_{z_i^{-1}}(f)(y) - f \ast \mu_S(y)$.  For fixed $y$, the variables $Z_i(y)$ are independent and have mean zero, so it follows by the Marcinkiewicz-Zygmund inequality and H{\"o}lder's inequality that
\begin{eqnarray*}
\| \sum_{i=1}^k{Z_i(y)}\|_{L^p(\mu_S^k)}^p &\leq &O(p)^{p/2}\int{\left(\sum_{i=1}^k{|Z_i(y)|^2}\right)^{p/2}d\mu_S^k}\\ & \leq & O(p)^{p/2}k^{p/2-1}\sum_{i=1}^k{\int{|Z_i(y)|^p}d\mu_S^k}.
\end{eqnarray*}
Summing over $y$ and interchanging the order of summation we get
\begin{equation}\label{eqn.khin}
\sum_{y \in G}{\| \sum_{i=1}^k{Z_i(y)}\|_{L^p(\mu_S^k)}^p} \leq O(p)^{p/2}k^{p/2-1}\int{\sum_{i=1}^k{\sum_{y \in G}{|Z_i(y)|^p}}d\mu_{S}^k}.
\end{equation}
On the other hand,
\begin{equation*}
\left(\sum_{y\in G}{|Z_i(y)|^p}\right)^{1/p} =\|Z_i\|_{\ell^p(G)} \leq \|\rho_{z_i^{-1}}(f)\|_{\ell^p(G)} + \|f \ast \mu_S\|_{\ell^p(G)} \leq 2\|f\|_{\ell^p(G)}
\end{equation*}
by the triangle inequality.  Dividing (\ref{eqn.khin}) by $k^p$ and inserting the above and the expression for the $Z_i$s we get that
\begin{equation*}
\int{\sum_{y \in G}{\left|\frac{1}{k}\sum_{i=1}^k{\rho_{z_i^{-1}}(f)(y)} - f \ast \mu_S(y)\right|^p}d\mu_S^k(z)}=O(pk^{-1}\|f\|_{\ell^p(G)}^2)^{p/2}.
\end{equation*}
Pick $k=O(\epsilon^{-2}p)$ such that the right hand side is at most $(\epsilon \|f\|_{\ell^p(G)}/4)^p$ and write $\mathcal{L}$ for the set of $x \in S\times \dots \times S$ (where the Cartesian product is $k$-fold) for which the integrand above is at most $(\epsilon \|f\|_{\ell^p(G)}/2)^p$; by averaging $\mu_S^k(\mathcal{L}^c) \leq 2^{-p}$ and so $\mu_S^k(\mathcal{L}) \geq 1-2^{-p} \geq 1/2$.

Now, $\Delta:=\{(t,\dots,t): t \in T\}$ has $\mathcal{L}\Delta  \subset ST\times \dots \times ST$, whence $|\mathcal{L} \Delta | \leq 2K^k|\mathcal{L}|$ and so
\begin{equation*}
\langle 1_\Delta \ast 1_{\Delta^{-1}},1_{\mathcal{L}^{-1}} \ast 1_{\mathcal{L}}\rangle_{\ell^2(G\times \dots \times G)} = \|1_\mathcal{L} \ast 1_\Delta \|_{\ell^2(G\times \dots \times G)}^2\geq |\Delta|^2|\mathcal{L}|/2K^k,
\end{equation*}
by the Cauchy-Schwarz inequality since the adjoint of $g \mapsto 1_{\mathcal{L}} \ast g$ is $g \mapsto 1_{\mathcal{L}^{-1}} \ast g$ and similarly for $g \mapsto g \ast 1_\Delta$.

By averaging it follows that at least $|\Delta|^2/2K^k$ pairs $(z,y) \in \Delta\times\Delta$ have $1_{\mathcal{L}^{-1}} \ast 1_{\mathcal{L}}(zy^{-1})>0$, and hence there is some $t \in T$ such that there is a set $X\subset Tt^{-1}$ of size at least $|T|/2K^k$ elements with $1_{\mathcal{L}^{-1}} \ast 1_{\mathcal{L}}(x,\dots,x)>0$ for all $x \in X$. 

Thus for each $x \in X$ there is some $z(x) \in \mathcal{L}$ and $y(x) \in \mathcal{L}$ such that $y(x)_i=z(x)_ix$.  But then by the triangle inequality we get that
\begin{eqnarray*}
\|\rho_{x^{-1}}(f \ast \mu_S) - f \ast \mu_S\|_{\ell^p(G)}& \leq &\|\rho_{x^{-1}}\left(\frac{1}{k}\sum_{i=1}^k{\rho_{z(x)_i^{-1}}(f)}\right) -  f \ast \mu_S\|_{\ell^p(G)}\\&&+\|\rho_{x^{-1}}\left(\frac{1}{k}\sum_{i=1}^k{\rho_{z(x)_i^{-1}}(f)} - f \ast \mu_S \right)\|_{\ell^p(G)}.
\end{eqnarray*}
However, since $\rho_x$ is isometric on $\ell^p(G)$ we see that
\begin{eqnarray*}
\|\rho_x(f \ast \mu_S) - f \ast \mu_S\|_{\ell^p(G)} &\leq & \|\frac{1}{k}\sum_{i=1}^k{\rho_{y(x)_i^{-1}}(f)}-  f \ast \mu_S\|_{\ell^p(G)}\\&&+\|\frac{1}{k}\sum_{i=1}^k{\rho_{z(x)_i^{-1}}(f)} - f \ast \mu_S\|_{\ell^p(G)},
\end{eqnarray*}
and we are done since $z(x),y(x) \in \mathcal{L}$.
\end{proof}
The important thing to note about the Croot-Sisask lemma is that the $p$-dependence of the size of the set $X$ is very good.  The natural Fourier analytic analogue (essentially given in \cite{bou::4}, and clearly exposited in \cite{sis::}) gives an exponentially worse bound.  To make use of this strength we use the aforementioned L{\'o}pez-Ross identity.
\begin{proof}[Proof of Proposition \ref{prop.key}]  We apply Lemma \ref{lem.cs} to the function $f:=1_{AS}$ and with the set $S^{-1}$ (so that $|S^{-1}T^{-1}| \leq L|S^{-1}|$) to get a set $X$ with $|X| \geq (2L)^{O(\epsilon^{-2}k^2p)}|T|$ such that
\begin{equation*}
\|\rho_x(1_{AS} \ast \mu_{S^{-1}}) - 1_{AS}\ast\mu_{S^{-1}}\|_{\ell^p(G)} \leq \frac{\epsilon  \|1_{AS}\|_{\ell^p(G)}}{ek} \textrm{ for all } x \in X.
\end{equation*}
Since $\rho$ is isometric on $\ell^p(G)$ and $\rho_{1_G}$ is the identity we may certainly assume that $X$ is a symmetric neighbourhood of the identity.  Furthermore, by the triangle inequality we have
\begin{equation*}
\|\rho_x(1_{AS} \ast \mu_{S^{-1}}) - 1_{AS}\ast\mu_{S^{-1}}\|_{\ell^p(G)} \leq \epsilon e^{-1} \|1_{AS}\|_{\ell^p(G)} \textrm{ for all } x \in X^k.
\end{equation*}
Now for any (real) function $g$ we have
\begin{equation*}
\mu_{A^{-1}} \ast g(x) - \mu_{A^{-1}} \ast g(1_G) = \mu_{A^{-1}} \ast (\rho_x(g) - g) (1_G) = \langle\mu_A, \rho_x(g)-g\rangle.
\end{equation*}
Thus by H{\"o}lder's inequality we have
\begin{equation*}
|\mu_{A^{-1}} \ast g(x) - \mu_{A^{-1}} \ast g(1_G) | \leq \|\mu_A\|_{\ell^{p'}(G)}\|\rho_x(g)-g\|_{\ell^p(G)}.
\end{equation*}
Putting $g=1_{AS} \ast \mu_{S^{-1}}$ we conclude that
\begin{eqnarray*}
|\mu_{A^{-1}} \ast 1_{AS} \ast \mu_{S^{-1}}(x) - \mu_{A^{-1}} \ast 1_{AS} \ast \mu_{S^{-1}}(1_G) | & \leq & \frac{\epsilon \|\mu_A\|_{\ell^{p'}(G)}\|1_{AS}\|_{\ell^p(G)}}{e}\\ & \leq & \frac{\epsilon |A|^{1/p'}|AS|^{1/p}}{e|A|}\leq \frac{\epsilon K^{1/p}}{e}
\end{eqnarray*}
for all $x \in X^k$.  Putting $p:=2+\log K$ we get the conclusion.
\end{proof}
\begin{proof}[Proof of Proposition \ref{prop.bst}]  We simply take $T=A$, $L=K$ and $\epsilon=1/2$ in Proposition \ref{prop.key}.
\end{proof}

\section{Basic properties of Bohr sets}\label{sec.bohrprop}

Following \cite{bou::1} we use a slight generalization of the traditional notion of Bohr set, letting the width parameter vary according to the character.  The advantage of this definition is that the meet of two Bohr sets in the lattice of Bohr sets is then just their intersection.

Throughout the section we let $G$ be a finite (compact) Abelian group.  A set $B$ is called a \emph{Bohr set} if there is a \emph{frequency set} $\Gamma$ of characters on $G$, and a \emph{width function} $\delta \in (0,2]^\Gamma$ such that
\begin{equation*}
B=\{x \in G: |1-\gamma(x)| \leq \delta_\gamma \textrm{ for
all }\gamma \in \Gamma\}.
\end{equation*}
Technically the same Bohr set can be defined by different frequency sets and width functions; we make the standard abuse that when we introduce a Bohr set we are implicitly fixing a frequency set and width function.

There is a natural way of dilating Bohr sets which will be of particular use to us.  For a Bohr set $B$ and $\rho \in \R^+$ we denote by $B_\rho$ the Bohr set with frequency set $\Gamma$ and width function\footnote{Technically width function $\gamma \mapsto \min\{\rho \delta_\gamma,2\}$.} $\rho\delta$ so that, in particular, $B=B_1$ and more generally $(B_\rho)_{\rho'} =B_{\rho\rho'}$.

Given two Bohr sets $B$ and $B'$ we define their \emph{intersection} to be the Bohr set with frequency set $\Gamma \cup \Gamma'$ and width function $\delta \wedge \delta'$.  A simple averaging argument (\emph{c.f.} \cite[Lemma 4.20]{taovu::} but also the end of Lemma \ref{lem.cs}) can be used to see that the intersection of several Bohr sets is large.
\begin{lemma}[Intersections of Bohr sets]\label{lem.inter}
Suppose that $(B^{(i)})_{i=1}^k$ is a sequence of Bohr sets.  Then
\begin{equation*}
\mu_G(\bigwedge_{i=1}^kB^{(i)}) \geq \prod_{i=1}^k{\mu_G(B^{(i)}_{1/2})}.
\end{equation*}
\end{lemma}
\begin{proof}
Let $\Delta:=\{(x,\dots,x) \in G^k: x \in G\}$ and $S:=B_{1/2}^{(1)}\times \dots \times B_{1/2}^{(k)}$.  Then
\begin{equation}\label{eqn.j}
\int{1_\Delta \ast 1_{-\Delta} 1_S \ast 1_{-S} d\mu_{G^k}} = \int{(1_\Delta \ast 1_S)^2d\mu_{G^k}} \geq \mu_{G^k}(\Delta)^2\mu_{G^k}(S)^2
\end{equation}
by Cauchy-Schwarz.  The integrand on the left hand side is at most $\mu_{G^k}(\Delta)\mu_{G^k}(S)$ and it is supported on the set of $x \in \Delta-\Delta=\Delta$ such that $1_S \ast 1_{-S}(x)>0$.  But if $1_S\ast 1_{-S}(y,\dots,y)>0$ then
\begin{equation*}
y \in \bigcap_{i=1}^k{(B_{1/2}^{(i)}-B_{1/2}^{(i)})} \subset \bigcap_{i=1}^k{B_1^{(i)}} = (\bigwedge_{i=1}^k{B^{(i)}})_1.
\end{equation*}
Hence
\begin{equation*}
\mu_{G^k}(\supp 1_\Delta \ast 1_{-\Delta} 1_S \ast 1_{-S}) \leq \mu_G((\bigwedge_{i=1}^k{B^{(i)}})_1)\mu_{G^k}(\Delta),
\end{equation*}
and inserting this in (\ref{eqn.j}) we get that
\begin{equation*}
\mu_G((\bigwedge_{i=1}^k{B^{(i)}})_1)\mu_{G^k}(\Delta)^2\mu_{G^k}(S) \geq \mu_{G^k}(\Delta)^2\mu_{G^k}(S)^2.
\end{equation*}
The result follows after some cancelation and noting that $\mu_{G^k}(S)$ is just the right hand side of the inequality in the statement of the lemma.
\end{proof}
Note that if $B$ is a Bohr set whose frequency set has one element, and whose width function is the constant function $2$ then there is an easy lower bound for $\mu_G(B_\eta)$ as the length of a certain arc on a circle:
\begin{equation}\label{eqn.bohrarc}
\mu_G(B_\eta) \geq \frac{1}{\pi}\arccos (1-2\eta^2) \geq \frac{1}{\pi}\min\{\eta,2\}.
\end{equation}
From this we immediately recover the usual lower bound on the size of a Bohr set with a larger frequency set from this and the preceding lemma.\footnote{To recover the bound in \cite[Lemma 4.20]{taovu::} some adjustments need to be made as our definition of a Bohr set is in terms of $\gamma(x)$ being close to $1$ rather than $\arg \gamma(x)$ being close to $0$.}

In \cite{bou::5} developed the idea of Bohr sets as approximate substitutes for groups, and since then his techniques have become an essential tool in additive combinatorics.  To begin with we define the \emph{entropy} of a Bohr set $B$ to be
\begin{equation*}
h(B):=\log \frac{\mu_G(B_{2})}{\mu_{G}(B_{1/2})}.
\end{equation*}
A trivial covering argument shows that $B_2$ can be covered by $\exp(h(B))$ translates of $B$, and if $B$ is actually a subgroup then $h(B)=0$.  It is often desirable to have a uniform bound on $h(B_\delta)$ for all $\delta \in (0,2]$, and such a bound is called the dimension of $B$ in other work.  Here, however, it is crucial that we do not insist on this.

We shall be particularly interested in Bohr sets which grow in a reasonably regular way because they will function well as approximate groups.  In light of the definition of entropy (which encodes growth over a fixed range) we say that a Bohr set $B$ is \emph{$C$-regular} if
\begin{equation*}
\frac{1}{1+Ch(B)|\eta|} \leq \frac{\mu_G(B_{1+\eta})}{\mu_G(B)} \leq 1+Ch(B)|\eta|
\end{equation*}
for all $\eta$ with $|\eta| \leq 1/Ch(B)$.  Crucially such Bohr sets are commonplace.
\begin{lemma}\label{lem.ubreg}
There is an absolute constant $C_\mathcal{R}$ such that if $B$ is a Bohr set then there is some $\lambda \in [1,2]$ such that $B_\lambda$ is $C_\mathcal{R}$-regular.
\end{lemma}
The proof is by a covering argument and follows, for example, \cite[Lemma 4.24]{taovu::}.  From now on we say that a Bohr set $B$ is \emph{regular} if it is $C_\mathcal{R}$-regular.

Finally, we write $\beta_\rho$ for the probability measure induced on $B_\rho$ by $\mu_G$, and $\beta$ for $\beta_1$.  These measures function as approximate analogues for Haar measure, and the following useful lemma of Green and Konyagin \cite{grekon::} shows how they can used to describe a sensible version of the annihilator of a Bohr set.
\begin{lemma}\label{lem.nest}
Suppose that $B$ is a regular Bohr set.  Then
\begin{equation*}
\{\gamma:|\wh{\beta}(\gamma)| \geq \kappa\} \subset \{\gamma: |1-\gamma(x)|=O(h(B)\kappa^{-1}\rho) \textrm{ for all } x \in B_\rho\}.
\end{equation*}
\end{lemma}
\begin{proof}
First, suppose that $|\wh{\beta}(\gamma)| \geq \kappa$ and $y \in B_\rho$.  Then
\begin{equation*}
|1-\gamma(y)|\kappa \leq |\int{\gamma(x)d\beta(x)} - \int{\gamma(x+y)d\beta(x)}|\leq \frac{\mu_G(B_{1+\rho}\setminus B_{1-\rho})}{\mu_G(B_1)} = O(h(B)\rho)
\end{equation*}
provided $\rho \leq 1/C_\mathcal{R}h(B)$.  The result is proved.
\end{proof}

\section{The large spectrum and Chang's theorem}\label{sec.relchang}

Given a probability measure $\mu$, a function $f \in L^1(\mu)$ and a parameter $\epsilon \in (0,1]$ we define the \emph{$\epsilon$-spectrum of $f$ w.r.t. $\mu$} to be the set
\begin{equation*}
\Spec_\epsilon(f,\mu):=\{\gamma \in \wh{G}: |(fd\mu)^\wedge(\gamma)| \geq \epsilon\|f\|_{L^1(\mu)}\}.
\end{equation*}
This definition extends the usual one from the case $\mu=\mu_G$.  We shall need a local version of a result of Chang \cite{cha::0} for estimating the `complexity' or `entropy' of the large spectrum.  

Given a set of characters $\Lambda$ and a function $\omega:\Lambda \rightarrow D:=\{z \in \C: |z| \leq 1\}$ we define
\begin{equation*}
p_{\omega,\Lambda}:=\prod_{\lambda \in \Lambda}{(1+\Re \omega(\lambda)\lambda)},
\end{equation*}
and call such a function a \emph{Riesz product for $\Lambda$}.  It is easy to see that all Riesz products are real non-negative functions.  They are at their most useful when they also have mass close to $1$: the set $\Lambda$ is said to be \emph{$K$-dissociated w.r.t. $\mu$} if
\begin{equation*}
\int{p_{\omega,\Lambda}d\mu} \leq \exp(K) \textrm{ for all } \omega:\Lambda \rightarrow D.
\end{equation*}
In particular, being $0$-dissociated w.r.t. $\mu_G$ is the usual definition of being dissociated.  This relativised version of dissociativity has a useful monotonicity property.
\begin{lemma}[Monotonicity of dissociativity]
Suppose that $\mu'$ is another probability measure, $\Lambda$ is $K$-dissociated w.r.t. $\mu$, $\Lambda' \subset \Lambda$ and $K' \geq K$.  Then $\Lambda'$ is $K'$-dissociated w.r.t. $\mu' \ast \mu$.
\end{lemma}
Conceptually the next definition is inspired by the discussion of quadratic rank Gowers and Wolf give in \cite{gowwol::}.  The \emph{$(K,\mu)$-relative entropy} of a set $\Gamma$ is the size of the largest subset $\Lambda\subset \Gamma$ such that $\Lambda$ is $K$-dissociated w.r.t. $\mu$.
\begin{lemma}[The Chang bound, {\cite[Lemma 4.6]{san::01}}]\label{lem.changbd}  Suppose that $0 \not \equiv f \in L^2(\mu)$ and write $L_f:=\|f\|_{L^2(\mu)}\|f\|_{L^1(\mu)}^{-1}$.  Then the set $\Spec_\epsilon(f,\mu)$ has $(1,\mu)$-relative entropy $O( \epsilon^{-2}\log 2L_f)$.
\end{lemma}
The proof of this goes by a Chernoff-type estimate, the argument for which follows \cite[Proposition 3.4]{greruz::0}, and then the usual argument of Chang from \cite{cha::0}.

Although Chang's theorem cannot be significantly improved (see \cite{gre::5}, and \cite{gre::6} for a discussion), there are some small refinements and discussions of their limitations in the work \cite{shk::,shk::2} and \cite{shk::4} of Shkredov.

Low entropy sets of characters are majorised by large Bohr sets, a fact encoded in the following lemma.  The proof is a minor variant of \cite[Lemma 6.3]{san::01}.
\begin{lemma}[Annihilating dissociated sets]\label{lem.majdissoc}
Suppose that $B$ is a regular Bohr set and $\Delta$ is a set of characters with $(\eta,\beta)$-relative entropy $k$.  Then there is a set $\Lambda$ of size at most $k$ and a $\rho=\Omega(\eta/(1+h(B))(k+\log2\eta^{-1}))$, such that for all $\gamma \in \Delta$ we have
\begin{equation*}
|1-\gamma(x)| =O(k \nu + \rho' \rho^{-1} h(B_{\rho})) \textrm{ for all }x \in B_{\rho'} \wedge B'_\nu, \rho',\nu  \in \R^+
\end{equation*}
where $B'$ is the Bohr set with constant width function $2$ and frequency set $\Lambda$.
\end{lemma}
\begin{proof}
Let $L:=\lceil \log_2 3^k2(k+1)\eta^{-1}\rceil$, the reason for which choice will become apparent, and define
\begin{equation*}
\beta^+:=\beta_{1+L{\rho}} \ast \beta_{\rho} \ast \dots \ast \beta_{\rho},
\end{equation*}
where $\beta_{\rho}$ occurs $L$ times in the expression.  By regularity (of $B$) we can pick some ${\rho}\in (\Omega(\eta/(1+h(B))L),1]$ such that $B_{\rho}$ is regular and we have the point-wise inequality
\begin{equation*}
\beta \leq \frac{\mu_G(B_{1+L{\rho}})}{\mu_G(B)} \beta^+ \leq (1+\eta/3)\beta^+.
\end{equation*}
It follows that if $\Lambda$ is $\eta/2$-dissociated w.r.t. $\beta^+$ then $\Lambda$ is $\eta$-dissociated w.r.t. $\beta$, and hence $\Lambda$ has size at most $k$.  From now on all dissociativity will be w.r.t. $\beta^+$. 

We put $\eta_i:=i\eta/2(k+1)$ and begin by defining a sequence of sets $\Lambda_0,\Lambda_1,\dots$ iteratively such that $\Lambda_i$ is $\eta_i$-dissociated.  We let $\Lambda_0:=\emptyset$ which is easily seen to be $0$-dissociated.  Now, suppose that we have defined $\Lambda_i$ as required.  If there is some $\gamma \in \Delta\setminus \Lambda_i$ such that $\Lambda_i \cup \{\gamma\}$ is $\eta_{i+1}$-dissociated then let $\Lambda_{i+1}:=\Lambda_i \cup \{\gamma\}$.  Otherwise, terminate the iteration.

Note that for all $i \leq k+1$, if the set $\Lambda_i$ is defined then it is certainly $\eta/2$-dissociated and so $|\Lambda_i| \leq k$.  However, if the iteration had continued for $k+1$ steps then $|\Lambda_{k+1}| >k$.  This contradiction means that there is some $i\leq k$ such that $\Lambda:=\Lambda_i$ is $\eta_i$-dissociated and $\Lambda_i \cup \{\gamma\}$ is not $\eta_{i+1}$-dissociated for any $\gamma \in \Delta \setminus \Lambda_i$.

It follows that we have a set $\Lambda$ of at most $k$ characters such that for all $\gamma \in \Delta \setminus \Lambda$ there is a function $\omega:\Lambda \rightarrow D$ and $\nu \in D$ such that
\begin{equation*}
\int{p_{\omega,\Lambda}(1+\Re \nu \gamma)d\beta^+} > \exp(\eta_{i+1}).
\end{equation*}

Now, suppose that $\gamma \in \Delta$.  If $\gamma \in \Lambda$ then the conclusion is immediate, so we may assume that $\gamma \in \Delta \setminus \Lambda$.  Then, since $\Lambda$ is $\eta_i$-dissociated, we see that
\begin{equation*}
|\int{p_{\omega,\Lambda}\overline{\gamma} d\beta^+}|  >  \exp(\eta_{i+1})- \exp(\eta_i) \geq \frac{\eta}{2(k+1)}.
\end{equation*}
Applying Plancherel's theorem we get that
\begin{equation*}
 \frac{\eta}{2(k+1)} \leq \left|\sum_{\lambda \in \Span(\Lambda)}{\wh{p_{\omega,\Lambda}}(\lambda)\wh{\beta^+}(\gamma-\lambda)}\right|\leq 3^k\sup_{\lambda \in \Span(\Lambda)}{|\wh{\beta_\rho}(\gamma-\lambda)|^L}.
\end{equation*}
Given the choice of $L$ there is some $\lambda \in \Span(\Lambda)$ such that $|\wh{\beta_{\rho}}(\gamma-\lambda)| \geq 1/2$.  By  Lemma \ref{lem.nest} we see that
\begin{equation*}
\gamma-\lambda \in \{\gamma': |1-\gamma'(x)| = O(\rho'' h(B_\rho)) \textrm{ for all }x \in (B_{\rho})_{\rho''}\}.
\end{equation*}
On the other hand, by the triangle inequality if $\lambda \in \Span(\Lambda)$ then
\begin{equation*}
\lambda \in \{\gamma': |1-\gamma'(x)| \leq k \nu \textrm{ for all }x \in B'_\nu\},
\end{equation*}
and the result follows from a final application of the triangle inequality.
\end{proof}

\section{Containment in a Bohr set}\label{sec.bohrcont}

The object of this section is to show the following result.
\begin{proposition}\label{prop.cont}
Suppose that $G$ is a finite (compact) Abelian group, $d \geq 1$ and $X$ is a finite subset of $G$ with $\mu_G(nX) \leq n^d\mu_G(X)$ for all $n \geq 1$ and $\kappa \in (0,1]$ is a parameter.  Then there is a regular Bohr set $B$ such that
\begin{equation*}
X-X \subset B_\kappa \textrm{ and } \mu_G(B_2) \leq \exp(O(d\log 2d\kappa^{-1}))\mu_G(X).
\end{equation*}
\end{proposition}
What is important here is that given a set of relative polynomial growth we have produced a Bohr set which contains the original set, and which has controlled growth over a fixed range of dilations.  Extending this range down to zero can be done but involves considerable additional work as well as being unnecessary for our arguments.

The next lemma is the key ingredient which provides us with an appropriate Bohr set.  The idea originates with Green and Ruzsa in \cite[Lemma 2.3]{greruz::0}, but the lemma we record is more obviously related to \cite[Proposition 4.39]{taovu::}.
\begin{lemma}\label{lem.grt}
Suppose that $G$ is a finite (compact) Abelian group, $A,S \subset G$ have
\begin{equation*}
\mu_G(A+S) \leq K\mu_G(A) \textrm{ and }|\wh{1_{A+S}}(\gamma)| \geq (1-\epsilon)\mu_G(A+S).
\end{equation*}
Then $|1-\gamma(s)| \leq  \sqrt{2^3K \epsilon}$ for all $s \in S-S$.
\end{lemma}
\begin{proof}
By hypothesis there is a phase $\omega \in S^1$ such that
\begin{equation*}
\int{1_{A+S}\omega\gamma d\mu_G} = |\wh{1_{A+S}}(\gamma)| \geq (1-\epsilon)\mu_G(A+S).
\end{equation*}
It follows that
\begin{equation*}
\int{1_{A+S}|1-\omega\gamma|^2d\mu_G} =2\int{1_{A+S}(1-\omega\gamma)d\mu_G} \leq 2\epsilon\mu_G(A+S),
\end{equation*}
and so if $y_0,y_1 \in S$ then
\begin{equation*}
\int{1_{A}|1-\omega\gamma(y_i)\gamma|^2d\mu_G} \leq
\int{1_{A+S}|1-\omega\gamma|^2d\mu_G} \leq 2\epsilon \mu_G(A+S).
\end{equation*}
However, the Cauchy-Schwarz inequality tells us that
\begin{equation*}
|1-\gamma(y_0-y_1)|^2 \leq 2(|1-\omega\gamma(y_0)\gamma(x)|^2
+|1-\omega\gamma(y_1)\gamma(x)|^2)
\end{equation*}
for all $x \in G$, whence
\begin{equation*}
\int{1_{A}|1-\gamma(y_0-y_1)|^2d\mu_G} \leq 2^3\epsilon \mu_G(A+S),
\end{equation*}
and the result follows.
\end{proof}
To prove the proposition we use an idea of Schoen from \cite{sch::0}, first introduced to Fre{\u\i}man-type problems by Green and Ruzsa in \cite{greruz::0}.  The essence is that if we have sub-exponential growth of a set then we can apply the Cauchy-Schwarz inequality and Parseval's theorem in a standard way to get a Fourier coefficient of very close to maximal value.
\begin{proof}[Proof of Proposition \ref{prop.cont}]
By the pigeonhole principle there is some $l=O(d\log 2d)$ such that $\mu_G(lX) \leq 2\mu_G((l-1)X)$.  We let $B'$ be the Bohr set with width function the constant function $1/2$ and frequency set $\Gamma:=\Spec_{1-\epsilon}(1_{lX})$ where we pick $\epsilon:=2^{-10}\kappa^2$.

It follows by Lemma \ref{lem.grt} applied to $A=(l-1)X$ and $S=X$ that
\begin{equation*}
|1-\gamma(x)| \leq \sqrt{2^3.2.\epsilon} = \kappa/8 \textrm{ for all } x \in X-X \textrm{ and } \gamma \in \Spec_{1-\epsilon}(1_{lX}),
\end{equation*}
and hence that $X-X \subset B_{\kappa/4}'$.

It remains to show that the Bohr set is not too large.  Begin by noting that
\begin{equation}\label{eqn.u}
\int{(1_{lX} ^{(k)})^2d\mu_G} \geq \frac{1}{\mu_G(k(lX))}\left(\int{1_{lX}^{(k)}d\mu_G}\right)^2 \geq \frac{\mu_G(lX)^{2k-1}}{(kl)^d},
\end{equation}
where $1_{lX}^{(k)}$ denotes the $k$-fold convolution of $1_{lX}$ with itself, and the inequality is Cauchy-Schwarz and then the hypothesis.  On the other hand, by Parseval's theorem
\begin{eqnarray*}
\sum_{\gamma \not \in \Spec_{1-\epsilon}(1_{lX})}{|\wh{1_{lX}}(\gamma)|^{2k}} &\leq& ((1-\epsilon)\mu_G(lX))^{2k-2}\sum_{\gamma \in \wh{G}}{|\wh{1_{lX}}(\gamma)|^2}\\ & \leq & \exp(-\Omega(k\kappa))\mu_G(lX)^{2k-1}\leq \frac{\mu_G(lX)^{2k-1}}{2(kl)^d}
\end{eqnarray*}
for some $k=O(d\kappa^{-1}\log 2d\kappa^{-1})$.  In particular, from (\ref{eqn.u}) we have that
\begin{equation*}
\sum_{\gamma \not \in \Spec_{1-\epsilon}(1_{lX})}{|\wh{1_{lX}}(\gamma)|^{2k}} \leq \frac{1}{2}\int{(1_{lX} ^{(k)})^2d\mu_G}.
\end{equation*}
It then follows from Parseval's theorem and the triangle inequality that
\begin{eqnarray*}
\sum_{\gamma \in \Spec_{1-\epsilon}(1_{lX})}{|\wh{1_{lX}}(\gamma)|^{2k}}& = &\sum_{\gamma \in \wh{G}}{|\wh{1_{lX}}(\gamma)|^{2k}} - \sum_{\gamma \not \in \Spec_{1-\epsilon}(1_{lX})}{|\wh{1_{lX}}(\gamma)|^{2k}}\\ & \geq & \int{(1_{lX} ^{(k)})^2d\mu_G}-\frac{1}{2}\int{(1_{lX} ^{(k)})^2d\mu_G} = \frac{1}{2}\int{(1_{lX} ^{(k)})^2d\mu_G}.
\end{eqnarray*}
On the other hand by the triangle inequality $|\wh{\beta'}(\gamma)| \geq 1/2$ if $\gamma \in \Gamma$ since $\delta \leq 1/2$, whence
\begin{equation*}
\sum_{\gamma \in \wh{G}}{|\wh{1_{lX}}(\gamma)|^{2k}|\wh{\beta'}(\gamma)|^2}\geq \frac{1}{4}\sum_{\gamma \in \Spec_{1-\epsilon}(1_{lX})}{|\wh{1_{lX}}(\gamma)|^{2k}} \geq \frac{\mu_G(lX)^{2k-1}}{8(kl)^d}.
\end{equation*}
But, by Parseval's theorem and H{\"o}lder's inequality we have that
\begin{eqnarray*}
\sum{|\wh{1_{lX}}(\gamma)|^{2k}|\wh{\beta'}(\gamma)|^2} & = & \int{(1_{lX} ^{(k)}\ast \beta')^2d\mu_G}\\ &\leq & \|1_{lX}^{(k)}\ast 1_{-lX}^{(k)}\|_{L^1(G)}\|\beta'\ast \beta'\|_{L^\infty(G)} = \frac{\mu_G(lX)^{2k}}{\mu_G(B')},
\end{eqnarray*}
and so
\begin{equation*}
\mu_G(B') \leq (kl)^d\mu_G(lX) \leq \exp(O(d\log 2d\kappa^{-1}))\mu_G(X).
\end{equation*}
Finally we apply Lemma \ref{lem.ubreg} to get a regular Bohr set $B$ with $B_2 \subset B_1'$ and $B_\kappa \supset B_{\kappa/4}'$ so the result is proved.
\end{proof}

\section{Covering and growth in Abelian groups}\label{sec.cov}

Covering lemmas are a major tool in additive combinatorics and have been since their development by Ruzsa in \cite{ruz::01}.  This was further extended by Green and Ruzsa in \cite{greruz::}, and they play a pivotal role in the non-Abelian theory as was highlighted by Tao in \cite{tao::6} which we do not have many other techniques.

While the most basic form of covering lemmas do work in the non-Abelian setting there is a refined argument due to Chang \cite{cha::0} which does not port over so easily.  
\begin{lemma}[Chang's covering lemma, {\cite[Lemma 5.31]{taovu::}}]
Suppose that $G$ is an (discrete) Abelian group and $A,S \subset G$ are finite sets with $|nA|\leq K^n|A|$ for all $n \geq 1$ and $|A+S| \leq L|S|$.  Then there is a set $T$ with $|T| = O(K \log 2KL)$ such that\footnote{Recall that $\Span(T):=\{\sum_{t \in T}{\sigma_t.t}: \sigma \in \{-1,0,1\}^{t}\}$.}
\begin{equation*}
A \subset \Span(T)+ S-S.
\end{equation*}
\end{lemma}
We shall also need the following slight variant which provides a way in Abelian groups to pass from relative polynomial growth on one scale to all scales.
\begin{lemma}[A variant of Chang's covering lemma]\label{lem.ccl}
Suppose that $G$ is an (discrete) Abelian group and $A,S \subset G$ are finite sets with $|kA + S| < 2^k|S|$. Then there is a set $T \subset A$ with $|T| < k$ such that $A \subset \Span(T) + S-S$.
\end{lemma}
\begin{proof}
Let $T$ be a maximal \emph{$S$-dissociated} subset of $A$, that is a maximal subset of $A$ such that
\begin{equation*}
(\sigma.T + S) \cap (\sigma'.T+S) = \emptyset \textrm{ for all } \sigma\neq \sigma' \in \{0,1\}^T.
\end{equation*}
Now suppose that $x' \in A \setminus T$ and write $T':=T \cup \{x'\}$. By maximality of $T$ there are elements $\sigma,\sigma' \in \{0,1\}^{T'}$ such that $(\sigma.T' +S) \cap (\sigma'.T' + S) \neq \emptyset$. Now if $\sigma_{x'} = \sigma'_{x'}$ then $(\sigma|_T.T +S) \cap (\sigma'|_T.T + A) \neq \emptyset$ contradicting the fact that $T$ is $S$-dissociated. Hence, without loss of generality, $\sigma_{x'}=1$ and $\sigma'_{x'}=0$, whence
\begin{equation*}
x' \in \sigma'|_T.T -\sigma|_T.T + S - S \subset \Span(T) + S-S.
\end{equation*}
We are done unless $|T|\geq k$; assume it is and let $T' \subset T$ be a set of size $k$. Denote $\{\sigma.T': \sigma \in \{0,1\}^{T'}\}$ by $P$ and note that $P \subset kA$, whence
\begin{equation*}
2^k|S| = |P+S| \leq |kA+S| < 2^k|S|.
\end{equation*}
This contradiction completes the proof.
\end{proof}
Although this is a result in Abelian groups, it has many parallels with Milnor's proof in \cite{mil::} establishing the dichotomy between polynomial growth and exponential growth in solvable groups.

The above lemma is particularly useful for controlling the order of relative polynomial growth through the next result, an idea which was introduced by Green and Ruzsa in \cite{greruz::}.
\begin{lemma}\label{lem.polyg}
Suppose that $G$ is an (discrete) Abelian group, $X \subset G$ and $2X-X \subset \Span(T) +X-X$ for some set $T$ of size $k$.  Then
\begin{equation*}
|(n+1)X-X| \leq (2n+1)^k|X-X| \textrm{ for all } n \geq 1.
\end{equation*}
\end{lemma}
\begin{proof}
By induction it is immediate that
\begin{equation*}
(n+1)X-X\subset n\Span(T) + X-X,
\end{equation*}
and it is easy to see that $|n\Span(T)| \leq (2n+1)^k$ from which the result follows.  
\end{proof}

\section{Lattices and coset progressions}\label{sec.lcp}

The geometry of numbers seems to play a pivotal role in proofs of Fre{\u\i}man-type theorems, and we direct the reader to \cite[Chapter 3.5]{taovu::} or the notes \cite{gre::1} for a much more comprehensive discussion.

Recall that $\Lambda$ is a \emph{lattice} in $\R^k$ if there are linearly independent vectors $v_1,\dots,v_k$ such that $\Lambda = v_1\Z+\dots+v_k\Z$; we call $v_1,\dots,v_k$ a basis for $\Lambda$.  Furthermore, a set $K$ in $\R^k$ is called a \emph{convex body} if it is convex, open, non-empty and bounded.

We require the following application of John's theorem and Minkowski's second theorem which provides us with a way of producing a generalised arithmetic progression from some sort of `convex progression'\footnote{A more formal notion of convex progression is introduced by Green in \cite{gre::1}, where a detailed discussion and literature survey may be found.}.

\begin{lemma}[{\cite[Lemma 3.33]{taovu::}}]\label{lem.tctvhv}  Suppose that $K$ is a symmetric convex body and $\Lambda$ is a lattice, both in $\R^d$.  Then there is a proper $d$-dimensional progression $P$ in $K \cap \Lambda$ such that $|P| \geq \exp(-O(d\log 2d)) |K \cap \Lambda|$.
\end{lemma}
The $\exp(-O(d\log d))$ factor should not come as a surprise: consider packing a $d$-dimensional cube (playing the role of the generalised progression) inside a $d$-dimensional sphere.

The question remains of how to find a `convex progression', and to do this Ruzsa \cite{ruz::9} introduced an important embedding.  Suppose that $G$ is a (discrete) finite Abelian group and $\Gamma \subset \wh{G}$.  Then we define a map
\begin{eqnarray*}
R_\Gamma:G& \rightarrow & C(\Gamma,\R)\\ x & \mapsto  & R_\Gamma(x):\Gamma \rightarrow \R; \gamma \mapsto \frac{1}{2\pi}\arg(\gamma(x)),
\end{eqnarray*}
where the argument is taken to lie in $(-\pi,\pi]$.  Note that $R_\Gamma$ preserves inverses, meaning that $R_\Gamma(-x)=-R_\Gamma(x)$, and furthermore if\footnote{Recall that if $X$ is a normed space then $\|\cdot\|_X$ denotes the norm on that space, so that $\|f\|_{C(\Gamma,\R)} = \|f\|_{L^\infty(\Gamma)}$.}
\begin{equation*}
\|R_\Gamma(x_1)\|_{C(\Gamma,\R)} + \dots + \|R_\Gamma(x_d)\|_{C(\Gamma,\R)} <1/2
\end{equation*}
then
\begin{equation*}
R_\Gamma(x_1+\dots + x_d)=R_\Gamma(x_1)+\dots+R_\Gamma(x_d).
\end{equation*}
This essentially encodes the idea that $R_\Gamma$ behaves like a Fre{\u\i}man morphism\footnote{We direct the unfamiliar reader to \cite[Chapter 5.3]{taovu::}.}.  We shall use this embedding to establish the following proposition.
\begin{proposition}\label{prop.coset}
Suppose that $G$ is a finite Abelian group, $d \in \N$ and $B$ is a Bohr set such that
\begin{equation*}
\mu_G(B_{(3d+1)\delta})<2^d\mu_G(B_\delta) \textrm{ for some } \delta < 1/4(3d+1).
\end{equation*}
Then $B_\delta$ contains a proper coset progression $M$ of dimension at most $d$ satisfying the estimate $\beta_\delta(M) =\exp(-O(d\log 2d))$.
\end{proposition}
\begin{proof}
We write $\Gamma$ for the frequency set of $B$ and note that we may assume that $L:=\bigcap{\{\ker \gamma: \gamma \in \Gamma\}}$ is trivial.  Indeed, if it is non-trivial we may quotient out by it without impacting the hypotheses of the proposition; we call the quotiented Bohr set $B'$ and note that $B_\delta=B_\delta'+L$ from which the result follows.

To start with note that if $x \in B_\eta$ then
\begin{equation*}
\|R_\Gamma(x)\|_{C(\Gamma,\R)} \leq \frac{1}{2\pi} \arccos (1-\eta^2/2) \leq 2\eta,
\end{equation*}
and so since $2(3d+1)\delta < 1/2$ we have that if $x_1,\dots,x_{3d+1} \in B_\delta$ then
\begin{equation}\label{eqn.vfr}
R_\Gamma(x_1+\dots+x_{3d+1}) = R_\Gamma(x_1)+\dots + R_\Gamma(x_{3d+1}).
\end{equation}
By hypothesis we then have that
\begin{eqnarray*}
|(3d+1)R_\Gamma(B_\delta)|  = |R_\Gamma((3d+1)B_\delta)| & \leq & |(3d+1)B_\delta|\\ & \leq & |B_{(3d+1)\delta}| < 2^d|B_\delta| = 2^d|R_\Gamma(B_\delta)|.
\end{eqnarray*}
Apply the variant of Chang's covering lemma in Lemma \ref{lem.ccl} to the set $R_\Gamma(B_\delta)$ (which is symmetric since $R_\Gamma$ preserves inverses and $B_\delta$ is symmetric) to get a set $X\subset R_\Gamma(B_\delta)$ with $|X| \leq d$ such that
\begin{equation*}
3R_\Gamma(B_\delta) \subset \Span(X) + 2R_\Gamma(B_\delta).
\end{equation*}
Writing $V$ for the real subspace of $C(\Gamma,\R)$ generated by $X$ we see that $\dim V \leq d$ and (by induction) that
\begin{equation*}
nR_\Gamma(B_\delta) \subset V+ 2R_\Gamma(B_\delta)
\end{equation*}
for all $n$.  Now, suppose that $v \in 2R_\Gamma(B_\delta)$.  It follows that 
\begin{equation*}
n.v \in 2nR_\Gamma(B_\delta) \subset V+ 2R_\Gamma(B_\delta).
\end{equation*}
for all naturals $n$.  Since $2R_\Gamma(B_\delta)$ is finite we see that there are two distinct naturals $n$ and $n'$ and some element $w \in 2R_\Gamma(B_\delta)$ such that $n.v, n'.v \in V+w$.  It follows that $(n-n').v \in V$ whence $v \in V$ since $V$ is a vector space and $n \neq n'$.  We conclude that $R_\Gamma(B_\delta) \subset V$.

Let $E$ be the group generated by $B_\delta$ which is finite, and note that $H:=R_\Gamma(E) + C(\Gamma,\Z)$ is a closed discrete subgroup of $C(\Gamma,\R)$, where $C(\Gamma,\Z)$ is the group of $\Z$-valued functions on $\Gamma$.  Since $H$ is a closed discrete subgroup of $C(\Gamma,\R)$ contained in $V$, it is also a closed discrete subgroup of $V$.  Since $V$ is certainly generated by $R_\Gamma(B_\delta)$ and $H \supset R_\Gamma(B_\delta)$ we see that $\Lambda:=H \cap V$ has finite co-volume and so is a lattice in $V$.

Let $\rho$ be the unique solution to $|1-\exp(2\pi i \rho)| = \eta$ in the range $[0,1/2]$, and write $Q_\rho$ for the $\rho$-cube in $C(\Gamma,\R)$, which is a symmetric convex body in $C(\Gamma,\R)$, and so $K:=V \cap Q_\rho$ is a symmetric convex body in $V$.  Now, by Lemma \ref{lem.tctvhv} the set $K\cap \Lambda$ contains a proper $d$-dimensional progression $P$ of size $\exp(-O(d\log 2d))|K\cap \Lambda|$.  

To see this note that by (\ref{eqn.vfr}), $R_\Gamma|_{B_\delta}$ is a Fre{\u\i}man $2$-homomorphism.  Now, if the elements $x_1,x_2,x_3,x_4 \in B_\delta$ have
\begin{equation*}
R_\Gamma(x_1)+R_\Gamma(x_2)=R_\Gamma(x_3)+R_\Gamma(x_4)
\end{equation*}
then
\begin{equation*}
R_\Gamma(x_1+x_2-x_3-x_4) = R_\Gamma(x_1)+R_\Gamma(x_2)+R_\Gamma(-x_3)+R_\Gamma(-x_4)=0.
\end{equation*}
However, $R_\Gamma(x)=0$ if and only if $\gamma(x)=1$ for all $\gamma \in \Gamma$, which is to say if and only if $x\in L$.  Since $L$ is trivial we conclude that $x_1+x_2=x_3+x_4$ and hence that $R_\Gamma$ is injective on $B_\delta$, and $R_\Gamma^{-1}:R_\Gamma(B_\delta)\rightarrow B_\delta$ is a Fre{\u\i}man $2$-homomorphism.

On the other hand, by (\ref{eqn.vfr}) $R_\Gamma:B_\delta \rightarrow R_\Gamma(B_\delta)$ is a Fre{\u\i}man $2$-homomorphism, and so $R_\Gamma:B_\delta \rightarrow R_\Gamma(B_\delta)$ is a Fre{\u\i}man $2$-isomorphism, and hence so is its inverse $R_\Gamma^{-1}:R_\Gamma(B_\delta)\rightarrow B_\delta$

Since $B_\delta=R_\Gamma^{-1}(K \cap \Lambda)$, we are done by, for example, \cite[Proposition 5.24]{taovu::}, which simply says that the image of a proper coset progression under a Fre{\u\i}man isomorphism of order at least $2$ is a proper coset progression of the same size and dimension; in particular $R_\Gamma^{-1}(P)$ is a proper coset progression of size $\exp(-O(d\log 2d))|B_\delta|$ and dimension at most $d$.
\end{proof}

\section{Proof of the main theorem}\label{sec.pf}

The result driving Theorem \ref{thm.main} is the following which brings together all the ingredients of the paper.
\begin{theorem}\label{thm.core}
Suppose that $G$ is a finite Abelian group, $A,S \subset G$ have $|A+S| \leq K\min\{|A|,|S|\}$, and $\epsilon \in (0,1]$ is a parameter.  Then there is a proper coset progression $M$ with
\begin{equation*}
\dim M =O(\epsilon^{-2}\log^62\epsilon^{-1}K) \textrm{ and } |M| \geq \left(\frac{\epsilon}{2\log K}\right)^{O(\epsilon^{-2}\log^62\epsilon^{-1}K)}|A+S|,
\end{equation*}
such that for any probability measure $\mu$ supported on $M$ we have
\begin{equation*}
\|1_{A+S}\ast \mu\|_{\ell^\infty(G)}\geq 1-\epsilon \textrm{ and } \|1_A \ast \mu\|_{\ell^\infty(G)} \geq (1-\epsilon)\frac{|A|}{|A+S|}.
\end{equation*}
\end{theorem}
\begin{proof}
We start by thinking of $G$ as discrete and using counting measure.  By Pl{\"u}nnecke's inequality \cite[Corollary 6.28]{taovu::} there is a non-empty set $S' \subset S$ such that
\begin{equation*}
|A+A+S'| \leq \left(\frac{K\min\{|A|,|S|\}}{|S|}\right)^2|S'|\leq K^2\frac{|A||S'|}{|S|}\leq K^2|A|.
\end{equation*}
Note, in particular, that since $|A+A+S'| \geq |A|$ we have $|S'| \geq |S|/K^2$ from the second inequality.  Applying the inequality again we get a non-empty set $A' \subset A$ such that
\begin{equation*}
|A'+(A+S')+(A+S')| \leq K^4|A'|,
\end{equation*}
and it follows that 
\begin{equation}\label{eqn.kj}
|(A+S') + (A+S')| \leq K^4|A+S'|.
\end{equation}
Now we apply Proposition \ref{prop.key} with $T=A$ to get a symmetric neighbourhood of the identity $X$ such that
\begin{equation*}
|X| \geq \exp(-O(\epsilon^{-2}k^2\log^2 2K))|A+S| 
\end{equation*}
since $|A| \geq |A+S|/K$, and
\begin{equation}\label{eqn.hg}
|\mu_{-A} \ast 1_{A+S'} \ast \mu_{-S'}(x) - 1| \leq \epsilon/4 \textrm{ for all }x \in kX.
\end{equation}
In the first instance it follows that $kX \subset (A+S')-(A+S')$.  On the other hand, by the Pl{\"u}nnecke-Ruzsa estimates \cite[Corollary 6.29]{taovu::} applied to (\ref{eqn.kj}) we have that
\begin{eqnarray*}
|4l((A+S')-(A+S'))| & \leq &K^{32l}|A+S'|\\ &= & \exp(O(l\log K +\epsilon^{-2}k^2\log^2 K))|X|,
\end{eqnarray*}
and hence
\begin{equation*}
|4lkX| \leq \exp(O(l\log 2K +\epsilon^{-2}k^2\log^2 2K))|X|.
\end{equation*}
We put $l=\lceil \epsilon^{-2}k^2\log 2K\rceil$ so that
\begin{equation*}
|(3kl+1)X| \leq |4klX| \leq 2^{kl.O(k^{-1} \log 2K)}|X|.
\end{equation*}
Hence we can pick $k$ such that
\begin{equation*}
1+\log \epsilon^{-1}K \leq k=O(\log 2\epsilon^{-1}K) \textrm{ and } |(3kl+1) X| <2^{kl}|X|.
\end{equation*}
By the variant of Chang's covering lemma in Lemma \ref{lem.ccl} there is some set $T$ of size at most $kl = O(\epsilon^{-2}\log^4 2 \epsilon^{-1}K)$ such that $3X \subset \Span(T) +2X$, and hence (by Lemma \ref{lem.polyg})
\begin{equation*}
|(n+2)X)| \leq n^{O(\epsilon^{-2}\log^4 2\epsilon^{-1}K)}|2X| \textrm{ for all } n \geq 1.
\end{equation*}
On the other hand $|2X| \leq 2^{kl}|X|$, and so (re-scaling the measure to think of $G$ as compact) we have
\begin{equation*}
\mu_G(nX) \leq n^{O(\epsilon^{-2}\log^4 2\epsilon^{-1}K)}\mu_G(X) \textrm{ for all } n \geq 1.
\end{equation*}
Now, by Proposition \ref{prop.cont} applied to the set $X$ there is a $d=O(kl\log 2kl\kappa^{-1})$ (which we may also assume is at least $1$) and a regular Bohr set $B$ such that
\begin{equation*}
X-X \subset B_{\kappa/2} \textrm{ and } \mu_G(B_2) \leq \exp(d)\mu_G(X).
\end{equation*}
Let $c$ be the absolute constant in the following technical lemma and note that since $X$ is a neighbourhood of the identity, $X \subset B$ and $\beta(X) \geq \exp(-d)$.

We apply Chang's theorem relative to $B$ to get that $\Spec_c(1_X,\beta)=\Spec_c(\mu_X)$ has $(1,\beta)$-relative entropy
\begin{equation*}
r=O(c^{-2}\log 2\|1_X\|_{L^2(\beta)}\|1_X\|_{L^1(\beta)}^{-1}) = O(d).
\end{equation*}
It follows from Lemma \ref{lem.majdissoc} that there is a set of characters $\Lambda$ of size $r$ and a $\rho = \Omega(1/(1+h(B))r)$ such that for all $\gamma \in \Spec_c(\mu_X)$ we have
\begin{equation*}
|1-\gamma(x)| =O(\nu r + \rho' r h(B)h(B_{\rho})) \textrm{ for all }x \in B_{\rho'} \wedge B_\nu',
\end{equation*}
where $B'$ is the Bohr set with width function the constant function $2$ and frequency set $\Lambda$. Provided $\rho \geq \kappa$ we see that
\begin{equation*}
\mu_G(X) \leq \mu_G(B_{\rho/2}) \leq \mu_G(B_{1/2}) \textrm{ and } \mu_G(B_{2\rho}) \leq \mu_G(B_2) \leq \exp(d)\mu_G(X),
\end{equation*}
and so it follows that $h(B),h(B_\rho) \leq d$.  It follows that $\rho=\Omega(1/d^2)$ and
\begin{equation*}
|1-\gamma(x)| =O(\nu d +  \rho' d^3 ) \textrm{ for all }x \in B_{\rho'} \wedge B_\nu' \textrm{ and } \gamma \in\Spec_c(\mu_X).
\end{equation*}
Pick $\rho'=\Omega(\epsilon/d^3K^2)$ and $\nu=\Omega(\epsilon/K^2d)$ such that $B'':=B_{\rho'} \wedge B_\nu'$ has
\begin{equation*}
|1-\gamma(x)| \leq \epsilon/4K^2 \textrm{ for all }x \in B'' \textrm{ and } \gamma \in\Spec_c(\mu_X).
\end{equation*}
In particular
\begin{equation*}
\rho',\nu = \Omega(1/K^{2}d^{O(1)}).
\end{equation*}
For each $\lambda \in \Lambda$ write $B^{(\lambda)}$ for the Bohr set with frequency set $\{\lambda\}$ and width function the constant function $2$, thus $B_\nu' = \bigwedge_{\lambda \in \Lambda}{B^{(\lambda)}_\nu}$.  By Lemma \ref{lem.inter} we see that
\begin{equation*}
\mu_G(B''_\eta) \geq \mu_G(B_{\eta \rho'/2})\prod_{\lambda \in \Lambda}{\mu_G(B_{\eta \nu/2}^{(\lambda)})}. 
\end{equation*}
On the other hand since $B^{(\lambda)}$ has a frequency set of size $1$ we see (from (\ref{eqn.bohrarc})) that
\begin{equation*}
\mu_G(B^{(\lambda)}_{\eta'}) \geq \frac{1}{\pi}\min\{\eta',2\}.
\end{equation*}
Now, if $\eta \rho'/2\geq \kappa$ then we have that
\begin{equation*}
\mu_G(B''_\eta) \geq  (\eta\nu/2\pi)^r\mu_G(X),
\end{equation*}
and on the other we have that $\mu_G(B) \leq \exp(d)\mu_G(X)$.  Let $t \geq 1$ be a natural such that
\begin{equation*}
(16\pi(3t+1) \nu^{-1})^r\exp(d) < 2^t \textrm{ and } t=O(d\log 2dK).
\end{equation*}
Then if $\eta \in [1/8(3t+1),1/4(3t+1))$ we have
\begin{equation*}
\mu_G(B''_{(3t+1)\eta}) <  2^t\mu_G(B''_\eta).
\end{equation*}
We now apply Proposition \ref{prop.coset} to get that $B_{\eta}'' $ contains a proper coset progression $M$ of dimension at most $t$ and size $(2t)^{-O(t)}\mu_G(X)$.  The result is proved on an application of the next lemma provided such a choice of $\eta$ is possible.  This can be done if $\kappa$ can be chosen such that
\begin{equation*}
\frac{\rho'}{8(3t+1)} > \kappa,
\end{equation*}
which can be done with $\kappa = \Omega(\epsilon^{O(1)}K^{-O(1)})$, and working this back gives that $t=O(\epsilon^{-2}\log^62\epsilon^{-1}K)$ and the result.
\end{proof}
The next lemma is here simply to avoid interrupting the flow of the previous argument, and the hypotheses are set up purely for that setting.  The proof is simply a series of standard Fourier manipulations.
\begin{lemma}\label{lem.tech}
There is an absolute constant $c>0$ such that if $G$ is a finite Abelian group, $A,S,X \subset G$ have $|A+S| \leq K\min\{|A|,|S|\}$, $S' \subset S$ has $|S'| \geq |S|/K^2$, $k \geq \log \epsilon^{-1}K$ is a natural number such that
\begin{equation*}
|\mu_{-A} \ast 1_{A+S'} \ast \mu_{-S'}(x) - 1| \leq \epsilon/4 \textrm{ for all }x \in kX,
\end{equation*}
and $M$ is a set such that
\begin{equation}\label{eqn.import}
|1-\gamma(x)| \leq \epsilon/4K^2 \textrm{ for all }x \in M \textrm{ and } \gamma \in \Spec_c(\mu_X),
\end{equation}
then for any probability measure $\mu$ supported on $M$ we have
\begin{equation*}
\|1_{A+S}\ast \mu\|_{\ell^\infty(G)}\geq 1-\epsilon \textrm{ and } \|1_A \ast \mu\|_{\ell^\infty(G)} \geq (1-\epsilon)\frac{|A|}{|A+S|}.
\end{equation*}
\end{lemma}
\begin{proof}
Integrating the first hypothesis we get that
\begin{equation*}
|\langle \mu_{-A} \ast 1_{A+S'} \ast \mu_{-S'}, \mu_X^{(k)}\rangle - 1| \leq \epsilon/4,
\end{equation*}
where $\mu_X^{(k)}$ denotes the $k$-fold convolution of $\mu_X$ with itself.  By Fourier inversion we have that
\begin{equation}\label{eqn.qsw}
\left|\sum_{\gamma \in \wh{G}}{\wh{1_{A+S'}}(\gamma)\overline{\wh{\mu_{A}}(\gamma)\wh{\mu_{S'}}(\gamma)\wh{\mu_X}(\gamma)^k}} -1\right| \leq \epsilon/4.
\end{equation}
The triangle inequality, Cauchy-Schwarz and Parseval's theorem in the usual way tell us that
\begin{eqnarray}\nonumber 
 \sum_{\gamma \in \wh{G}}{|\wh{1_{A+S'}}(\gamma)\wh{\mu_A}(\gamma)\wh{\mu_{S'}}(\gamma)|} & \leq & \mu_G(A+S')\|\wh{\mu_A}\|_{\ell^2(\wh{G})}\|\wh{\mu_{S'}}\|_{\ell^2(\wh{G})}\\ \label{eqn.fr} & = & \frac{\mu_G(A+S')}{\sqrt{\mu_G(A)\mu_G(S')}} \leq K^2.
\end{eqnarray}

Then, by the triangle inequality for any probability measure $\mu$ supported on $M$ we have that
\begin{equation}\label{eqn.intest}
|\wh{\mu}(\gamma) -1| \leq \epsilon/4K^2 \textrm{ for all } \gamma \in \Spec_c(\mu_X).
\end{equation}
We conclude that 
\begin{equation*}
E:=|\langle 1_{A+S'}\ast \mu,\mu_A \ast \mu_{S'} \ast \mu_X^{(k)} \ast \mu\rangle -1| = \left|\sum_{\gamma \in \wh{G}}{\wh{1_{A+S'}}(\gamma)\wh{\mu}(\gamma)\overline{\wh{\mu_{A}}(\gamma)\wh{\mu_{S'}}(\gamma)\wh{\mu_X}(\gamma)^k\wh{\mu}(\gamma)}} -1\right|
\end{equation*}
is at most $S_1+S_2+S_3$ where
\begin{equation*}
S_1:= \left|\sum_{\gamma \not \in \Spec_c(\mu_X)}{\wh{1_{A+S'}}(\gamma)\overline{\wh{\mu_{A}}(\gamma)\wh{\mu_{S'}}(\gamma)\wh{\mu_X}(\gamma)^k}(|\wh{\mu}(\gamma)|^2-1)}\right|,
\end{equation*}
\begin{equation*}
S_2:= \left|\sum_{\gamma \in \Spec_c(\mu_X)}{\wh{1_{A+S'}}(\gamma)\overline{\wh{\mu_{A}}(\gamma)\wh{\mu_{S'}}(\gamma)\wh{\mu_X}(\gamma)^k}(|\wh{\mu}(\gamma)|^2-1)}\right|,
\end{equation*}
and
\begin{equation*}
S_3:= \left|\sum_{\gamma \in \wh{G}}{\wh{1_{A+S'}}(\gamma)\overline{\wh{\mu_{A}}(\gamma)\wh{\mu_{S'}}(\gamma)\wh{\mu_X}(\gamma)^k}}-1\right|.
\end{equation*}
By the triangle inequality and (\ref{eqn.fr}) we see that
\begin{equation*}
S_1 \leq \sup_{\gamma \not \in \Spec_c(\mu_X)}{|\wh{\mu_X}(\gamma)|^k}.\sum_{\gamma \in \wh{G}}{|\wh{1_{A+S'}}(\gamma)\wh{\mu_A}(\gamma)\wh{\mu_{S'}}(\gamma)|}  \leq c^kK^2 \leq \epsilon/4
\end{equation*}
for a suitable choice of $c=\Omega(1)$ since $k\geq \log \epsilon^{-1}K$;  by (\ref{eqn.fr}) and (\ref{eqn.intest}) we see that
\begin{equation*}
S_2 \leq 2\sup_{\gamma \in \Spec_c(\mu_X)}{|\wh{\mu}(\gamma)-1|}.\sum_{\gamma \in \wh{G}}{|\wh{1_{A+S'}}(\gamma)\wh{\mu_A}(\gamma)\wh{\mu_{S'}}(\gamma)|}  \leq 2(\epsilon/4K^2).K^2 \leq \epsilon/2;
\end{equation*}
and finally by (\ref{eqn.qsw}) we see that $S_3 \leq \epsilon/4$, so that $E \leq \epsilon$.  It follows from this that
\begin{equation*}
\langle 1_{A+S'}\ast \mu, \mu_A \ast \mu_{S'} \ast \mu_X^{(k)} \ast \mu\rangle \geq 1-\epsilon,
\end{equation*}
and hence by averaging that
\begin{equation*}
\|1_{A+S'} \ast \mu\|_{L^\infty(G)} \geq 1-\epsilon \textrm{ and } \|1_{A} \ast \mu\|_{L^\infty(G)} \geq (1-\epsilon)\frac{\mu_G(A)}{\mu_G(A+S')}.
\end{equation*}
The lemma is proved.
\end{proof}
It is worth making a couple of remarks before continuing.  First, Theorem \ref{thm.core} can be extended to infinite Abelian groups by embedding the sets there in a finite group via a sufficiently large Fre{\u\i}man isomorphism.  This is the finite modelling argument of Green and Ruzsa \cite[Lemma 2.1]{greruz::0}, but we shall not pursue it here.

The expected $\epsilon$-dependence in Theorem \ref{thm.core} may be less clear than the $K$-dependence.  The argument we have given works equally well for the so-called popular difference set in place of $1_{A+S}$, that is the set
\begin{equation*}
D(A,S):=\{x \in G: 1_A \ast 1_S(x)\geq c\epsilon/K\}
\end{equation*}
for sufficiently small $c$.  On the other hand Wolf, in \cite{wol::0}, develops the Niveau set construction of Ruzsa \cite{ruz::4,ruz::6}, to show that even finding a large sumset in such popular difference sets is hard, and it seems likely that her arguments can be adapted to cover the case of $D(A,S)$ containing a proportion $1-\epsilon$ of a sumset.

Understanding this, even in the model setting of $G=\F_2^n$, would be of great interest since a better $\epsilon$-dependence would probably yield better analysis of inner products of the form $\langle 1_A \ast 1_S,1_T\rangle$ which are of importance in, for example, Roth's theorem \cite{rot::0,rot::}. 

We are now in a position to prove Theorem \ref{thm.main} by an easy pigeonhole argument.
\begin{proof}[Proof of Theorem \ref{thm.main}]  Fre{\u\i}man $2$-embed the sets $A$ and $S$ into a finite group (via, for example, the method of \cite[Lemma 2.1]{greruz::0}); if we can prove the result there then it immediately pulls back.

Apply Theorem \ref{thm.core} with $\epsilon=1/2(1+\sqrt{2})$ to get a proper $d$-dimensional coset progression $M$.  Note that we may assume the progression is symmetric by translating it and possibly shrinking it by a factor of $\exp(d)$; this has no impact on the bounds.  Thus we put
\begin{equation*}
M=H+\{x_1.l_1+\dots+x_d.l_d: |l_i| \leq L_i \textrm{ for all }1 \leq i \leq d\}
\end{equation*}
where $L_1,\dots, L_d \in \N$, $H \leq G$ and $x_1,\dots,x_d \in G$. Write
\begin{equation*}
M_\eta:=H+\{x_1.l_1+\dots+x_d.l_d: |l_i| \leq \eta L_i \textrm{ for all }1 \leq i \leq d\},
\end{equation*}
and note that $|M_1| \leq \exp(O(d))|M_{1/2}|$.  On the other hand if $j \eta \leq 1/2$ then we have
\begin{equation*}
M_{1/2} \subset M_{1/2+\eta} \subset \dots \subset M_{1/2+j\eta} = M_1,
\end{equation*}
and so it follows that there is some $\eta = \Omega(1/d)$ and $i\leq j = O(d)$ such that
\begin{equation*}
|M_{1/2+i\eta}| \leq 2^{1/2}|M_{1/2+(i-1)\eta}|.
\end{equation*}
Since $\eta = \Omega(1/d)$ we easily have that $|M_\eta| = \exp(-O(d\log d))|M_1|$.  On the other hand if we apply the conclusion of Theorem \ref{thm.core} with
\begin{equation*}
\mu=\frac{1_{M_{1/2+i\eta}}+1_{M_{1/2+(i-1)\eta}}}{|M_{1/2+i\eta}|+|M_{1/2+(i-1)\eta}|}
\end{equation*}
we get an element $x$ such that
\begin{equation*}
|(x+A+S)\cap M_{1/2+i\eta}|+|(x+A+S)\cap M_{1/2+(i-1)\eta}|
\end{equation*}
is at least
\begin{equation*}
(1-\epsilon)(|M_{1/2+i\eta}|+|M_{1/2+(i-1)\eta}|).
\end{equation*}
But then if $z \in M_{\eta}$ we get that
\begin{eqnarray*}
1_{A+S} \ast 1_{-(A+S)}(z) & = & 1_{x+A+S} \ast 1_{-(x+A+S)}(z)\\ & \geq & 1_{(x+A+S) \cap M_{1/2+i\eta}}\ast 1_{-(x+A+S)\cap M_{1/2+(i-1)\eta}}(z)\\ & \geq & |(x+A+S) \cap M_{1/2+i\eta}| + |z+((x+A+S)\cap M_{1/2+(i-1)\eta})|\\ & & - |((x+A+S) \cap M_{1/2+i\eta}) \cup (z+((x+A+S)\cap M_{1/2+(i-1)\eta}))|\\ & \geq & |(x+A+S)\cap M_{1/2+i\eta}|+|(x+A+S)\cap M_{1/2+(i-1)\eta}| - |M_{1/2+i\eta}|\\ & \geq & (1-(1+\sqrt{2})\epsilon)|M_{1/2+(i-1)\eta}|>0,
\end{eqnarray*}
and it follows that $(A-A)+(S-S)$ contains $M_\eta$.  Tracking through the bounds we get the result.
\end{proof}

\section{Concluding remarks and applications}\label{sec.appln}

To begin with we should remark that in the case when $G$ has bounded exponent or is torsion-free, we can get slightly better bounds and the argument is much simpler because of the presence of a good modelling lemmas.  In the first case we get the following result, a proof of which (in the case $G=\F_2^n$) is contained in the appendix as it is so short.
\begin{theorem}[Bogolyubov-Ruzsa Lemma for bounded exponent Abelian groups]\label{thm.bogexp}
Suppose $G$ is an Abelian group of exponent $r$ and $A,S \subset G$ are finite non-empty sets such that $|A+S| \leq K\min\{|A|,|S|\}$.  Then $(A-A)+(S-S)$ contains a subspace $V$ of size $\exp(-O_r(\log^42K))|A+S|$. 
\end{theorem}
In the second, the material of \S\S\ref{sec.bohrprop}--\ref{sec.lcp} can be replaced similar but more standard arguments because of the following modelling lemma.
\begin{lemma}[Modelling for torsion-free Abelian groups, {\cite[Theorem 3.5]{ruz::03}}]
Suppose that $G$ is a torsion-free Abelian group, $A \subset G$ is a finite non-empty set and $k \geq 2$ is a natural.  Then for every $q \geq |kA-kA|$ there is a set $A' \subset A$ with $|A'| \geq |A|/k$ such that $A'$ is Fre{\u\i}man $k$-isomorphic to a subset of $\Z/q\Z$.
\end{lemma}
\begin{theorem}[Bogolyubov-Ruzsa lemma for torsion-free Abelian groups]
Suppose that $G$ is a torsion-free Abelian group and $A,S \subset G$ are finite non-empty sets such that $|A+S| \leq K\min\{|A|,|S|\}$.  Then $(A-A)+(S-S)$ contains a proper symmetric $d(K)$-dimensional coset progression $M$ of size $\exp(-h(K))|A+S|$.  Moreover, we may take $d(K)=O(\log^42K)$ and $h(K)=O(\log^42K\log2\log 2K)$.
\end{theorem}
Returning to Theorem \ref{thm.main} it is easy to see that we must have $d(K),h(K)=\Omega(\log K)$ by considering a union of $\sqrt{K}$ coset progressions of dimension $\log_2 \sqrt{K}$, and even achieving this bound may be hard without refining the definition of a coset progression.  (See the comments of Green in \cite{tao::5} for a discussion of this.)

The paper \cite{sch::1} was a major breakthrough in proving the first good bounds for (a slight variant of) Theorem \ref{thm.main}; it was essentially shown that one could take
\begin{equation*}
d(K),h(K)=O(\exp(O(\sqrt{\log K})))
\end{equation*}
for torsion-free or bounded-exponent Abelian groups.

Indeed, it should be clear that while we do not use \cite{sch::1} directly in the proof of Theorem \ref{thm.main}, it has had a considerable influence on the present work and the applications which now follow are from the end of that paper as well.

\subsection*{Fre{\u\i}man's theorem} As an immediate corollary of Theorem \ref{thm.main} and Chang's covering lemma we have the following.
\begin{theorem}[Fre{\u\i}man's theorem for Abelian groups]  Suppose that $G$ is an (discrete) Abelian group and $A \subset G$ is finite with $|A\pm A| \leq K|A|$.  Then $A$ is contained in a $d(K)$-dimensional coset progression $M$ of size at most $\exp(h(K))|A|$.  Moreover, we may take $d(K),h(K)=O(K\log^{O(1)}2K)$.
\end{theorem}
By considering a union of $K$ dissociated translates of a coset progression it is easy to see that we must have $d(K),h(K)=\Omega(K)$, so the result is close to best possible.

Green and Ruzsa in \cite{greruz::0} provided the first bounds of $d(K),h(K)=O(K^{4+o(1)})$, and the peppering of their work throughout this paper should indicate the importance of their ideas.

Schoen in \cite{sch::1} improved the bounds to $O(K^{3+o(1)})$ and to $O(K^{1+o(1)})$ for certain classes of groups, and in \cite{cwasch::} the structure is further elucidated with particular emphasis on getting good control on the dimension.

\subsection*{The $U^3$-inverse theorem} Theorem \ref{thm.main} can be inserted into the various $U^3$-inverse theorems of Tao and Green \cite{gretao::1} for finite Abelian groups of odd order, and Samorodnitsky \cite{sam::} (see also \cite{wol::}) for $\F_2^n$ to improve the bounds there.  In particular one gets the following.
\begin{theorem}[$U^3(\F_2^n)$-inverse theorem]
Suppose that $f \in L^\infty(\F_2^n)$ has $\|f\|_{U^3(\F_2^n)} \geq \delta \|f\|_{L^\infty(\F_2^n)}$.  Then there is a quadratic polynomial $q:\F_2^n \rightarrow \F_2$ such that
\begin{equation*}
|\langle f,(-1)^q\rangle_{L^2(\F_2^n)}| \geq \exp(-O(\log^{O(1)}2\delta^{-1}))\|f\|_{L^\infty(\F_2^n)}.
\end{equation*}
\end{theorem}
In fact the connection between good bounds in results of this type and good bounds in Fre{\u\i}man-type theorems is quite clearly developed by Green and Tao in \cite{gretao::6} and Lovett in \cite{lov::}.

\subsection*{Long arithmetic progressions in sumsets} The question of finding long arithmetic progressions in sets of integers is one of central interest in additive combinatorics.  The basic question has the following form: suppose that $A_1,\dots,A_k \subset \{1,\dots,N\}$ all have density at least $\alpha$.  How long an arithmetic progression can we guarantee that $A_1+\dots+A_k$ contain?

For one set this is addressed by the notoriously difficult Szemer{\'e}di's theorem \cite{sze::,sze::0} where the best quantitative work is that of Gowers \cite{gow::4,gow::0}; for two sets the longest progression is much longer with the state of the art due to Green \cite{gre::0}; for three sets or more the results get even stronger with the work of Fre{\u\i}man, Halberstam and Ruzsa \cite{frehalruz::}; and finally for eight sets or more, longer again by the recent work of Schoen \cite{sch::1}.

Theorem \ref{thm.main} yields an immediate improvement for the case of four sets or more.
\begin{theorem}
Suppose that $A_1,\dots,A_4 \subset \{1,\dots,N\}$ all have density at least $\alpha$.  Then $A_1+\dots+A_4$ contains an arithmetic progression of length $N^{O(\log^{-O(1)}2\alpha^{-1})}$.
\end{theorem}
\begin{proof}
Since $|A_i+A_j| \leq 2\alpha^{-1}|A_i|$ for all $i,j$ we have, by averaging, that there is a symmetric set $A$ of density $\alpha^{O(1)}$ such that $A_1,\dots,A_4$ each contains a translate of $A$. In particular, the longest progression in $A-A+A-A$ is contained in a translate of $A_1+A_2+A_3+A_4$. 

Now, by Theorem \ref{thm.main} the set $A-A+A-A$ contains an $O(\log^{O(1)}\alpha^{-1})$-dimensional coset progression $M$ of size $\exp(-O(\log^{O(1)}\alpha^{-1}))N$.  Since $\Z$ is torsion-free the progression is just a generalised progression which certainly contains a $1$-dimensional progression of length $|M|^{1/\dim M}$.  The result is proved.
\end{proof}
It is not clear that this result gives the best possible conclusion for $k$ sets as $k$ tends to infinity, although if one were interested in this no doubt some improvement could be squeezed out by delving into the main proof.

\subsection*{$\Lambda(4)$-estimate for the squares} Inserting Theorem \ref{thm.main} into the work of \cite{cha::1} (itself developed from an argument of Bourgain in \cite{johlin::}) yields the following $\Lambda(4)$-estimate for the squares.
\begin{theorem}
Suppose that $n_1,\dots,n_k$ are naturals.  Then
\begin{equation*}
\int{\left|\sum_{i=1}^k{\exp(2\pi i n_i^2\theta)}\right|^4d\theta} =O(k^3\exp(-\Omega(\log^{\Omega(1)} 2k))).
\end{equation*}
\end{theorem}
This is essentially equivalent to inserting Theorem \ref{thm.main} into the proof of \cite[Theorem 8]{sch::1} and Gowers' \cite{gow::4} version of the Balog-Szemer{\'e}di Lemma \cite{balsze::}.  In any case a conjecture of Rudin \cite{rud::0} suggests that the bound $O(k^{2+o(1)})$ is likely to be true, and the above is not even a power-type improvement on the trivial upper bound of $k^3$.

\subsection*{The Konyagin-{\L}aba theorem} Theorem \ref{thm.main} inserted into the argument at the end of \cite{sch::1} yields the following quantitative improvement to a result from \cite{konaba::}.
\begin{theorem}[Konyagin-{\L}aba theorem]
Suppose that $A$ is a set of reals and $\alpha \in \R$ is transcendental.  Then 
\begin{equation*}
|A+\alpha.A| = \exp(\Omega(\log^{\Omega(1)}2|A|))|A|.
\end{equation*}
\end{theorem}
What is particularly interesting here is that there is a simple construction which shows that there are arbitrarily large sets $A$ with $|A+\alpha.A| = \exp(O(\sqrt{\log |A|}))|A|$.

\section*{Acknowledgement}

The author should like to thank Julia Wolf for useful discussions surrounding the $U^3(\F_2^n)$-inverse theorem, and an anonymous referee for a thorough reading of the paper and numerous useful suggestions.

\appendix

\section{Proof of Theorem \ref{thm.bogexp}}

Our objective in this appendix is to prove the following result.
\begin{theorem}\label{thm.here}
Suppose that $G:=\F_2^n$, and $A \subset G$ has density $\alpha>0$.  Then there is a subspace $V \leq G$ with $\cod V=O(\log^42\alpha^{-1})$ such that $V \subset 4A$.
\end{theorem}
We have distilled this argument out because it is short and just uses the two ingredients of the Croot-Sisask lemma and Chang's theorem.  For the reader interested in a little more motivation the sketch after the introduction may be of more interest.

In the rather special setting of $\F_2^n$ it is known from work of Green and Ruzsa \cite[Proposition 6.1]{greruz::0} that if $|A+A| \leq K|A|$ then $A$ is Fre{\u\i}man $8$-isomorphic to a set $A'$ of density $K^{-O(1)}$ in some $\F_2^m$, from which we get the following corollary of Theorem \ref{thm.here}.
\begin{corollary}
Suppose that $G:=\F_2^n$, and $A \subset G$ has $|A+A| \leq K|A|$.  Then there is a subspace $V \leq G$ with $|V| \geq \exp(-O(\log^42K))|A|$ such that $V \subset 4A$.
\end{corollary}

In this setting the Croot-Sisask lemma is the following.
\begin{lemma}[Croot-Sisask]  Suppose that $G:=\F_2^n$, $f \in L^p(G)$ and $A \subset G$ has density $\alpha>0$.  Then there is an $a \in A$ and a set $T$ with $\mu_G(T) \geq (\alpha/2)^{O(\epsilon^{-2}p)}$ such that
\begin{equation*}
\|\rho_t(f \ast \mu_A) - f \ast \mu_A\|_{L^p(G)} \leq \epsilon \|f\|_{L^p(G)} \textrm{ for all } t \in T.
\end{equation*}
\end{lemma}
Additionally Chang's theorem is as follows.
\begin{lemma}[Chang's theorem]  Suppose that $G:=\F_2^n$ and $A \subset G$ has density $\alpha>0$.  Then
\begin{equation*}
\cod \Spec_\epsilon(\mu_A)^\perp = O(\epsilon^{-2}\log 2\alpha^{-1}).
\end{equation*}
\end{lemma}
\begin{proof}[Proof of Theorem \ref{thm.here}]
We begin by noting that
\begin{equation}\label{eqn.earlier}
\langle 1_{2A}\ast 1_{A},1_{A} \rangle= \langle 1_{2A},1_A \ast 1_A \rangle= \alpha^2.
\end{equation}
By the Croot-Sisask lemma applied with $f:=1_{2A}$ we get a set $T \subset G$ with $\mu_G(T) \geq (\alpha/2)^{O(k^2p)}$ such that
\begin{equation*}
\|\rho_t(1_{2A} \ast 1_{A}) - 1_{2A} \ast 1_{A}\|_{L^p(G)} \leq \alpha/4ke \textrm{ for all }t \in T.
\end{equation*}
By the triangle inequality this gives
\begin{equation*}
\| \rho_t(1_{2A} \ast 1_A) - 1_{2A}\ast 1_A\|_{L^p(G)} \leq \alpha/4e \textrm{ for all } t \in kT,
\end{equation*}
and so on integrating (and applying the triangle inequality again) we have
\begin{equation*}
\| 1_{2A} \ast 1_A\ast \mu_T ^{(k)}- 1_{2A}\ast 1_A\|_{L^p(G)} \leq \alpha/4e.
\end{equation*}
By H{\"o}lder's inequality we get that
\begin{equation*}
|\langle 1_{2A} \ast 1_A\ast \mu_T^{(k)} ,1_{A} \rangle - \langle 1_{2A} \ast 1_A, 1_{A}\rangle| \leq \alpha\alpha^{1+1/(p-1)}/4e.
\end{equation*}
Choosing $p=1+\log \alpha^{-1}$ and inserting (\ref{eqn.earlier}) we have that
\begin{equation*}
|\langle 1_{2A} \ast 1_A\ast \mu_T^{(k)} ,1_{A} \rangle -\alpha^2| \leq \alpha^2/4,
\end{equation*}
and so by the triangle inequality
\begin{equation*}
\langle 1_{2A} \ast 1_A \ast \mu_T ^{(k)},1_{A}\rangle_{L^p(G)} \geq 3\alpha^2/4.
\end{equation*}
Now, put $V:=\Spec_{1/2}(\mu_T)^\perp$ and $g:=1_{2A} \ast 1_A \ast \mu_T ^{(k)}$, so that
\begin{eqnarray*}
\left|\langle g ,1_{A}\rangle-\langle g \ast \mu_V,1_{A}\rangle\right|= 
\left|\sum_{\gamma \not \in V^\perp }{\wh{1_{2A}}(\gamma)|\wh{1_A}(\gamma)|^2\wh{\mu_T}(\gamma)^k}\right|\leq \alpha 2^{-k} \leq \alpha^2/8,
\end{eqnarray*}
by Parseval's theorem, the definition of $V$ and by taking $k=O(\log2\alpha^{-1})$ a sufficiently large natural.  It follows by the triangle inequality that
\begin{equation*}
\langle 1_{2A} \ast 1_A \ast \mu_T^{(k)}\ast \mu_V ,1_{A}\rangle > \alpha^2/2,
\end{equation*}
and so, by averaging, that $\|1_{2A} \ast \mu_V\|_{L^\infty(G)}>1/2$.  We conclude that $4A$ contains $V$ by the pigeon-hole principle and the result is proved on applying Chang's theorem to see that
\begin{equation*}
\cod V = O(\log 2\mu_G(T)^{-1})=O(\log^42\alpha^{-1}).
\end{equation*}
\end{proof}

\bibliographystyle{halpha}

\bibliography{references}

\end{document}